\documentclass[11pt,a4paper]{amsart}

\usepackage{amsmath,amsthm,amsfonts,amssymb,mathrsfs,bbm}

\theoremstyle{plain}
\newtheorem{theorem}{Theorem}[section]
\newtheorem{lemma}[theorem]{Lemma}
\newtheorem{corollary}[theorem]{Corollary}

\theoremstyle{definition}

\newtheorem{example}[theorem]{Example}

\DeclareMathOperator{\dimh}{dim_H}

\allowdisplaybreaks[3]

\title[mass transference principle and sets with large intersections]{A mass transference principle and sets with large intersections}
\author{Tomas Persson}

\address{T.~Persson, Centre for Mathematical Sciences, Lund
  University, Box~118, 221~00~Lund, Sweden}

\email{tomasp@maths.lth.se}

\date{\today}


\thanks{I thank Henna Koivusalo for interesting discussions, as
  well as Bao-Wei Wang for his useful comments and kind
  hospitality. I am also grateful to Demi Allen for her comments,
  and for telling me about the paper by S.~Jaffard
  \cite{jaffard}, which was also communicated to me by Henna
  Koivusalo. Finally, I thank a referee for putting my attention
  to the papers by J.-M.~Aubry and S.~Jaffard \cite{aubryjaffard}
  and by A.~Durand \cite{durand1, durand3}.}

\begin{document}

\begin{abstract}
  I prove a mass transference principle for general shapes,
  similar to a recent result by H.~Koivusalo and M.~Rams. The
  proof relies on Vitali's covering lemma and manipulations with
  Riesz energies.

  The main novelty is that it is proved that the obtained
  limsup-set belongs to the classes of sets with large
  intersections, as defined by K.~Falconer. This has previously
  not been proved for as general shapes as in this paper.
\end{abstract}

\maketitle

\section{The mass transference principle}

The mass transference principle is an important tool in the
metric theory of numbers, and in other areas, when one is
interested in the Hausdorff dimension of the superior limit of a
sequence of sets, that is sets of the form
\[
E = \limsup_{j \to \infty} E_j = \bigcap_{k=1}^\infty \bigcup_{j
  = k}^\infty E_j.
\]

What was first called the mass transference principle was proved
by V.~Beresnevich and S.~Velani \cite{beresnevichvelani}. In a
simplified form, it is the following statement. Suppose that $B
(x_j, r_j)$ is a sequence of balls in $\mathbbm{R}^d$ such that
the Lebesgue measure $\lambda$ of $E = \limsup_{j \to \infty} B
(x_j, r_j)$ is full, i.e.\ $\lambda (\mathbbm{R}^d \setminus E) =
0$. Let $s \leq d$ and consider the set $E_s = \limsup_{j \to
  \infty} B (x_j, r_j^{d/s}) \subset E$, where the original balls
$B(x_j,r_j)$ are replaced by shrunken balls of radius
$r^{d/s}$. Then, for any ball $B$, the $s$-dimensional Hausdorff
measure of $B \cap E_s$ satisfies
\[
\mathscr{H}^s (B \cap E_s) = \mathscr{H}^s (B).
\]
In particular, $\dimh E_s \geq s$.

Actually, the mass transference principle mentioned above is not the
first result of this type. S.~Jaffard proved a one dimensional mass
transference principle a few years earlier \cite[Second Theorem~1 on
  Page~335, not the one on Page~333]{jaffard}.

The mass transference principles of S.~Jaffard, V.~Beresnevich and
S.~Velani, and other variations are more general than mentioned above,
giving information about Hausdorff measures with more general gauge
functions than $r \mapsto r^t$, $0 < t < d$. The reader if referred to
the papers by S.~Jaffard, and by V.~Beresnevich and S.~Velani for more
details.

Although it is possible to adjust the methods of this paper to
more general gauge functions than $r \mapsto r^t$, we will stay
simple and only consider these basic gauge functions.

There has recently been several interesting extensions and
variations of the mass transference principle, and this paper
shall not be the last. A recent survey article is the paper by
D.~Allen and S.~Troscheit \cite{allentroscheit}.

B.-W.~Wang, J.~Wu and J.~Xu \cite{wangwuxu} replaced the balls in
the mass transference principle by rectangles. A further
development in this direction is a recent result by B.-W.~Wang
and J.~Wu \cite{wangwu}.

D.~Allen and V.~Beresnevich \cite{allenberesnevich} proved a mass
transference principle for shrinking neighbourhoods of
$l$-dimensional subspaces. This was further developed by D.~Allen
and S.~Baker \cite{allenbaker}, who proved a mass transference
principle for shrinking neighbourhoods of sets of much more
general form.

In a recent paper, H.~Koivusalo and M.~Rams \cite{koivusalorams},
considered a set $E = \limsup_{j \to \infty} B (x_j,r_j)$, but
instead of shrinking the balls $B (x_j, r_j)$, they replaced them
by open subsets $U_j \subset B(x_j, r_j)$, and obtained a lower
bound on the Hausdorff dimension of the set $E_U = \limsup_{j \to
  \infty} U_j$.  The goal of this paper is to prove a similar
result, using different methods. In short, we prove a result of
the following form, see Theorem~\ref{the:main} for
details. Suppose that $\limsup B(x_j, r_j)$ has full Lebesgue
measure, and let $U_j \subset B(x_j, r_j)$ be open sets. Then the
Hausdorff dimension of $\limsup U_j$ can be estimated from below
by an expression involving the Lebesgue measures and Riesz
energies with respect to Lebesgue measure of the sets $U_j$.

The lower bound on the Hausdorff dimension obtained in this paper
is not the same as the one obtained by H.~Koivusalo and M.~Rams.
The bound on the dimension obtained in this paper is never better
than that obtained by H.~Koivusalo and M.~Rams, but they coincide
in some natural cases, for instance when the open sets $U_j$ are
balls or ellipsoids.

The method used in this paper also proves that the set $\limsup_{j \to
  \infty} U_j$ has a large intersection property, as introduced by
K.~Falconer \cite{falconer1, falconer2}. This implies among other
things that the Hausdorff dimension of countable intersections of sets
of the form $\limsup_{j \to \infty} U_j$ is the minimum of the
Hausdorff dimensions of the sets intersected. This intersection
property has already been proved in the case when $U_j$ are balls by
J.-M.~Aubry and S.~Jaffard \cite{aubryjaffard}.  The intersection
property has also been proved in a similar setting by A.~Durand
\cite{durand1, durand2, durand3}, for so called ubiquitous systems
when the sets forming the limsup-set are balls. The results of
A.~Durand also cover Hausdorff measures with more general gauge
functions. However, the result of this paper, that the set $\limsup_{j
  \to \infty} U_j$ has a large intersection property for more general
shapes than balls, seems new.

The paper is organised as follows. In the next section, we give
some background on sets with large intersections, and we state a
lemma which, together with Vitali's covering lemma, is the main
tool in this paper. In Section~\ref{sec:results}, we state the
mass transference principle obtained in this paper, and give some
corollaries. The proof of the mass transference principle is in
Section~\ref{sec:proof}, and in the Appendix, the above mentioned
lemma is proved.

\section{Sets with large intersections}

The classes of sets with large intersections were defined by
K.~Falconer \cite{falconer1, falconer2}. See also the paper by
Y.~Bugeaud \cite{bugeaud}.

A dyadic cube $D$ is a set of the form
\[
D = [ k_1 2^{-n}, (k_1 + 1) 2^{-n}) \times [ k_2 2^{-n}, (k_2 +
    1) 2^{-n}) \times \ldots \times [ k_d 2^{-n}, (k_d + 1)
      2^{-n})
\]
where $k_1, k_2, \ldots, k_d$ and $n$ are integers. We define the
outer content
\[
\mathscr{M}_\infty^s (E) = \inf \biggl\{\, \sum_k |D_k|^s : E
\subset \bigcup_{k} D_k, \ D_k \text{ are dyadic cubes.} \,
\biggr\}.
\]
The $s$-dimensional class of sets with large intersections is
denoted by $\mathscr{G}^s$ and can be defined as
\begin{multline*}
  \mathscr{G}^s = \bigl\{\, E \subset \mathbbm{R}^d : E \text{ is
    a $G_\delta$-set with } \mathscr{M}_\infty^t (E \cap D) \geq
  |D|^t \\ \text{for all dyadic cubes } D \text{ and for all } t
  < s \, \bigr\}.
\end{multline*}
For simplicity, we shall work in the $d$ dimensional torus
$\mathbbm{T}^d = \mathbbm{R}^d / \mathbbm{Z}^d$ instead of
$\mathbbm{R}^d$. We define
\begin{multline*}
  \mathscr{G}^s (\mathbbm{T}^d) = \bigl\{\, E \subset
  \mathbbm{T}^d : E \text{ is a $G_\delta$-set with }
  \mathscr{M}_\infty^t (E \cap D) \geq |D|^t \\ \text{for all
    dyadic cubes } D \text{ and for all } t < s \, \bigr\}.
\end{multline*}
Equivalently we may define $\mathscr{G}^s (\mathbbm{T}^d)$ to be
the family of sets $E$ such that $\pi^{-1} (E) \in
\mathscr{G}^s$, where $\pi \colon \mathbbm{R}^d \to
\mathbbm{T}^d$ is the projection $\pi (x) = x \mod 1$.

There are several other ways to define $\mathscr{G}^s$, see the
papers by K.~Falconer \cite{falconer1, falconer2}. Most
properties of $\mathscr{G}^s$ are easily carried over to the
corresponding statements for $\mathscr{G}^s (\mathbbm{T}^d)$, for
instance through the projection $\pi$ mentioned above.

It is immediately clear from the definition of $\mathscr{G}^s
(\mathbbm{T}^d)$ above that whenever $E \in \mathscr{G}^s
(\mathbbm{T}^d)$, then $\dimh E \geq s$. The class $\mathscr{G}^s
(\mathbbm{T}^d)$ also has the property that if $E_k \in
\mathscr{G}^s (\mathbbm{T}^d)$ for $k = 1, 2, 3, \ldots$, then
\[
\bigcap_k E_k \in \mathscr{G}^s (\mathbbm{T}^d).
\]
For this and other properties of $\mathscr{G}^s (\mathbbm{T}^d)$,
the reader is referred again to the papers by K.~Falconer
\cite{falconer1, falconer2}.

The following lemma will be important for the proof of the main
result. It is a slight variation of a lemma in
\cite{perssonreeve}. We give a proof in the Appendix.

\begin{lemma} \label{lem:frostman}
  Let $E_n$ be open sets in $\mathbbm{T}^d$ and let $\mu_n$ be
  measures with $\mu_n (\mathbbm{T}^d \setminus E_n) = 0$ and
  such that $\mu_n$ is absolutely continuous with respect to
  Lebesgue measure.

  If there is a constant $C$ such that
  \begin{equation} \label{eq:densitybounds}
    C^{-1} \leq \liminf_{n \to \infty} \frac{\mu_n (B)}{\lambda
      (B)} \leq \limsup_{n \to \infty} \frac{\mu_n (B)}{\lambda
      (B)} \leq C
  \end{equation}
  for any ball $B$, and
  \begin{equation} \label{eq:energybound}
    \iint |x-y|^{-s} \, \mathrm{d} \mu_n (x) \mathrm{d} \mu_n (y)
    < C
  \end{equation}
  for all $n$, then $\limsup\limits_{n \to \infty} E_n \in
  \mathscr{G}^s (\mathbbm{T}^d)$, and in particular $\dimh
  \limsup\limits_{n \to \infty} E_n \geq s$.
\end{lemma}

\section{New (and old) results} \label{sec:results}

Recall that the Lebesgue measure on $\mathbbm{T}^d$ is denoted by
$\lambda$. We define the $t$-dimensional Riesz energy of a set
$U$ by
\[
I_t (U) = \int_U \int_U |x-y|^{-t} \, \mathrm{d} x \mathrm{d} y,
\]
where $0 < t < d$. (As is customary, we write $\mathrm{d}x$
instead of $\mathrm{d} \lambda (x)$.)

\begin{theorem} \label{the:main}
  Let $(B (x_j, r_j))_{j=1}^\infty$ be a sequence of balls in
  $\mathbbm{T}^d$ with $r_j \to 0$, and let $(U_j)_{j=1}^\infty$
  be a sequence of open sets such that $U_j \subset B(x_j,
  r_j)$. Let
  \begin{equation} \label{eq:formulafors}
    s = \sup \biggl\{\, t > 0: \sup_{j} \frac{I_t (U_j) \lambda
      (B(x_j, r_j))}{\lambda (U_j)^2 } < \infty \, \biggr\}.
  \end{equation}

  Suppose that the set
  $E = \limsup_{j \to \infty} B (x_j,r_j)$
  has full Lebesgue measure. Then the set
  \[
  E_U = \limsup_{j \to \infty} U_j
  \]
  satisfies $\dimh E_U \geq s$ and $E_U \in \mathscr{G}^s
  (\mathbbm{T}^d)$.
\end{theorem}

It is now timely to make a comparison between
Theorem~\ref{the:main} and the result of H.~Koivusalo and
M.~Rams. Let $t > 0$. Consider the set $U_j$ and define a measure
$\eta$ by
\[
\eta (A) = \int_{A \cap U_j} \biggl( \int_{U_j} |x - y|^{-t} \,
\mathrm{d} y \biggr)^{-1} \, \mathrm{d} x.
\]
Then by Jensen's inequality
\begin{align*}
  \eta (A) &\leq \int_{A \cap U_j} \biggl( \int_{A \cap U_j} |x -
  y|^{-t} \, \frac{\mathrm{d} y}{\lambda(U_j)} \biggr)^{-1} \,
  \frac{\mathrm{d}x}{\lambda (U_j)} \\ & \leq \int_{A \cap U_j}
  \int_{A \cap U_j} |x - y|^{t} \, \frac{\mathrm{d}
    y}{\lambda(U_j)} \, \frac{\mathrm{d} x}{\lambda (U_j)} \leq
  |A|^t,
\end{align*}
and
\begin{align*}
  \eta (U_j) &= \int_{U_j} \biggl( \int_{U_j} |x - y|^{-t} \,
  \frac{\mathrm{d} y}{\lambda(U_j)} \biggr)^{-1} \,
  \frac{\mathrm{d} x}{\lambda (U_j)} \\ & \geq \biggl( \int_{U_j}
  \int_{U_j} |x-y|^{-t} \, \frac{\mathrm{d} x}{\lambda (U_j)}
  \frac{\mathrm{d} y}{\lambda (U_j)} \biggr)^{-1} = \frac{\lambda
    (U_j)^2}{I_j (U_j)}.
\end{align*}
Now, if $(U_{j,k})_{k=1}^\infty$ is a cover of $U_j$, then
\[
\sum_{k=1}^\infty |U_{j,k}|^t \geq \sum_{k=1}^\infty \eta
(U_{j,k}) \geq \eta (U_j) \geq \frac{\lambda(U_j)^2}{I_t (U_j)},
\]
which shows that $\frac{\lambda (U_j)^2}{I_t (U_j)} \leq
\mathscr{H}_\infty^t (U_j)$.  Hence, the condition
\[
\frac{I_t (U_j) \lambda (B(x_j,r_j))}{\lambda (U_j)^2 } < K,
\qquad \text{for all } j
\]
implies that
\[
\lambda (B (x_j, r_j)) \leq K \mathscr{H}_\infty^t (U_j), \qquad
\text{for all } j
\]
but these two conditions are not equivalent. The condition of
H.~Koivusalo and M.~Rams is that
\[
\lambda (B(x_j, r_j)) \leq \phi^s (U_j), \qquad \text{for all }
j,
\]
where the generalised singular value function $\phi^s$ (defined
in \cite{koivusalorams}) is such that $\phi^s (U_j) /
\mathscr{H}_\infty^s (U_j)$ is bounded and bounded away from
zero. Hence the dimension bound of Theorem~\ref{the:main} is
never better than that H.~Koivusalo and M.~Rams, but they
coincide when $U_j$ are for instance balls or ellipsoids, see
below.

An immediate corollary of Theorem~\ref{the:main} is the following
result, following a similar observation in \cite{fengetal}: By
replacing each $U_j$ with an open subset $V_j \subset U_j$, we
get a smaller limsup-set, but the obtained lower bound on the
dimension is nevertheless sometimes larger.

\begin{corollary}
  Let
  \[
  s = \sup \biggl\{\, \sup \biggl\{\, t : \sup_{j} \frac{I_t (V_j)
      \lambda (B(x_j, r_j))}{\lambda (V_j)^2} < \infty \,
    \biggr\} : V_i \subset U_i \, \biggr\}.
  \]
  Then with the notation and assumptions of
  Theorem~\ref{the:main} we have $\dimh E_U \geq s$ and $E_U \in
  \mathscr{G}^s (\mathbbm{T}^d)$.
\end{corollary}

The classical case is that the sets $U_n$ are balls. In this case
Theorem~\ref{the:main} gives the following corollary. This result has
previously been proved by J.-M.~Aubry and S.~Jaffard
\cite[Proposition~5.4]{aubryjaffard} when $d=1$. Closely related are
also the results by A.~Durand on large intersection properties and
ubiquitous systems \cite{durand1, durand2}, and the so called large
intersection transference principle \cite[Theorem~6.10]{durand3} for
ubiquitous systems. The results by A.~Durand are valid for Hausdorff
measures with more general gauge functions than the basic $r \mapsto
r^t$ considered in this paper.

\begin{corollary} \label{cor:balls}
  Let $(x_j)_{j=1}^\infty$ be a sequence of points in
  $\mathbbm{T}^d$ and suppose that the set $E = \limsup_{j \to
    \infty} B (x_j, r_j)$ has full Lebesgue measure. Then the set
  $E_U = \limsup_{j \to \infty} B (x_j, r_j^{d/s})$, where $0
  \leq s \leq d$, satisfies $\dimh E_U \geq s$ and $E_U \in
  \mathscr{G}^s (\mathbbm{T}^d)$.
\end{corollary}


It is possible to get a similar result when $U_j$ are ellipsoids.
Suppose that $U$ is an ellipsoid with semi-axes $\lambda_1 \geq
\lambda_2 \geq \ldots \geq \lambda_d$. Define the singular value
function by $\phi^s (U) = \lambda_1 \ldots \lambda_m
\lambda_{m+1}^{s-m}$, where $m$ is such that $m < s \leq m+1$. It
follows from a lemma by Falconer \cite[Lemma~2.2]{falconeraffine}
that $I_t (U) \leq K \lambda(U)^2 / \phi^t (U)$.  We obtain the
following corollary, of which Corollary~\ref{cor:balls} is a
special case. 

\begin{corollary} \label{cor:ellipses}
  Suppose $\limsup_{j \to \infty} B(x_j,r_j)$ has full measure,
  that the sets $U_j \subset B(x_j,r_j)$ are ellipsoids and let
  \[
  s = \sup \bigl\{\, t : \sup_j \lambda (B(x_j,r_j)) / \phi^t
  (U_j) < \infty \,\bigr\}.
  \]
  Then $\dimh \limsup\limits_{j \to \infty} U_j \geq s$ and
  $\limsup\limits_{j \to \infty} U_j \in \mathscr{G}^s
  (\mathbbm{T}^d)$.
\end{corollary}

Note that instead of ellipses, we can let $U_j$ be rectangles,
and Corollary~\ref{cor:ellipses} holds also in this case. In this
case the dimension result has previously been proved by Wang, Wu
and Xu \cite{wangwuxu}, but the result on large intersection
property is new in this case.

\begin{example}
  As an example to Corollary~\ref{cor:ellipses}, consider the set
  \begin{multline*}
    W (\tau) = W(\tau_1, \tau_2, \ldots, \tau_d) = \Bigl\{ \,
    (x_1, x_2, \ldots, x_d) \in \mathbbm{T}^d : \\ \max_j \{
    q^{\tau_j} \lVert q x_j \rVert \} < 1 \text{ for infinitely
      many } q \in \mathbbm{N} \, \Bigr\}.
  \end{multline*}
  Wang, Wu and Xu showed \cite[Corollary~5.1]{wangwuxu} how to
  use their mass transference principle to obtain that
  \[
  \dimh W (\tau) = D(\tau) := \min \biggl\{\, \frac{d + 1 + j \tau_j -
    \sum_{i=1}^j \tau_i}{1 + \tau_j} : 1 \leq j \leq d \,
  \biggr\}
  \]
  provided that $\frac{1}{d} \leq \tau_1 \leq \tau_2 \leq \ldots
  \leq \tau_d$. (The ordering is only important to give the right
  formula.) With aid of Corollary~\ref{cor:ellipses} we obtain
  that $W (\tau) \in \mathscr{G}^s (\mathbb{T}^d)$ with $s = D
  (\tau)$. Hence we may conclude that
  \[
  \dimh \bigcap_{\tau \in T} W(\tau) = \inf_{\tau \in T}
  D(\tau),
  \]
  if $T \subset [\frac{1}{d}, \infty)^d$ is an at most countable set.
\end{example}

\section{Proof of Theorem~\ref{the:main}} \label{sec:proof}

We shall use Riesz energies of sets and measures to prove the
result through Lemma~\ref{lem:frostman}. The $t$-dimensional
Riesz energy of a measure $\mu$ is defined by
\[
I_t (\mu) = \iint |x-y|^{-t} \, \mathrm{d} \mu(x) \mathrm{d} \mu
(y).
\]
Recall that we have defined the $t$-dimensional Riesz energy of a
set $U$ by
\[
I_t (U) = \int_U \int_U |x-y|^{-t} \, \mathrm{d} x \mathrm{d} y.
\]
We will also use the number
\[
J_t (U, V) = \int_U \int_V |x-y|^{-t} \, \mathrm{d} x \mathrm{d}
y.
\]

There will be an opportunity to use Vitali's covering lemma, see
for instance Evans and Gariepy \cite{evansgariepy} for a proof.

\begin{lemma}[Vitali's covering lemma] \label{lem:vitali}
  Let $\{\, B(x_j, r_j) : j \in \mathscr{I} \,\}$ be a collection
  of balls. Then there is a countable set $\mathscr{K} \subset
  \mathscr{I}$ such that $\{\, B(x_j, r_j) : j \in \mathscr{K}
  \,\}$ is a disjoint collection and
  \[
  \bigcup_{j \in \mathscr{I}} B(x_j, r_j) \subset \bigcup_{j \in
    \mathscr{K}} B(x_j, 5 r_j).
  \]
\end{lemma}

We are now ready to give the proof of Theorem~\ref{the:main}. For
technical reasons in the proof we will assume that there is a
number $c \in (0,1)$ such that $U_j \in B(x_j, c r_j)$. If this
is not the case, we can just replace every ball $B(x_j, r_j)$ by
a ball $B(x_j, 2 r_j)$ with twice as large diameter, which does
not influence any of the assumptions in the
theorem.\footnote{Note that it is therefore not necessary to
  assume that $U_j \subset B(x_j, r_j)$ in
  Theorem~\ref{the:main}; It is enough to assume that $U_j
  \subset B(x_j, \kappa r_j)$ for some $\kappa$ that does not
  depend on $j$.}

Since $E$ has full Lebesgue measure, for every $n$ there exists a
number $m_n$ such that the set
\[
E_n = \bigcup_{j=n}^{m_n} B (x_j, r_j)
\]
has Lebesgue measure $\lambda (E_n) > 1 - 1/n$.

For each $n$, we let $\mathscr{I}_n = \{n, n+1, \ldots,
m_n\}$. By Vitali's covering lemma, Lemma~\ref{lem:vitali}, there
is a set $\mathscr{K}_n \subset \mathscr{I}_n$ such that the
balls $B (x_j, r_j)$, $j \in \mathscr{K}_n$, are disjoint and
such that
\[
E_n \subset \bigcup_{j \in \mathscr{K}_n} B (x_j, 5 r_j).
\]
Note that since the balls are disjoint, we have
\begin{equation} \label{eq:sumofmeasure}
  \sum_{j \in \mathscr{K}_n} \lambda(B(x_j,r_j)) < 1.
\end{equation}

We put
\[
\tilde{E}_n = \bigcup_{j \in \mathscr{K}_n} B (x_j, r_j) \subset
E_n.
\]
Then $\lambda (\tilde{E}_n) < 1$ and
\begin{align*}
  \lambda (\tilde{E}_n) = \sum_{j \in \mathscr{K}_n} \lambda
  (B(x_j, r_j)) &= \sum_{j \in \mathscr{K}_n} 5^{-d} \lambda
  (B(x_j, 5 r_j)) \\ &\geq 5^{-d} \lambda (E_n) > 5^{-d} \Bigl( 1
  - \frac{1}{n} \Bigr).
\end{align*}

Define measures $\nu_n$ with support in the closure of
$\tilde{E}_n$ by
\[
\nu_n = \frac{1}{\lambda (\tilde{E}_n)} \lambda|_{\tilde{E}_n}.
\]
Then $\nu_n (\mathbbm{T}^d) = 1$ and $\frac{\mathrm{d}
  \nu_n}{\mathrm{d} \lambda} \leq \bigl(1 - \frac{1}{n}
\bigr)^{-1} \cdot 5 ^d$ which implies that
\[
\limsup_{n \to \infty} \frac{\nu_n (B)}{\lambda (B)} \leq 5^d
\]
for any ball $B$.  Letting $\mathscr{K}_n (B) = \{\, j \in
\mathscr{K}_n : x_j \in B \,\}$ we have
\begin{align*}
  \liminf_{n \to \infty} \frac{\nu_n (B)}{\lambda (B)} &\geq
  \liminf_{n \to \infty} \frac{\lambda (B \cap
    \tilde{E}_n)}{\lambda (B)} \geq \liminf_{n \to \infty}
  \frac{\displaystyle \sum_{j \in \mathscr{K}_n (B)} \lambda
    (B(x_j,r_j))}{\lambda (B)} \\ & = \liminf_{n \to \infty}
  \frac{\displaystyle \sum_{j \in \mathscr{K}_n (B)} 5^{-d}
    \lambda (B(x_j,5 r_j))}{\lambda (B)} \geq \liminf_{n \to
    \infty} 5^{-d} \frac{\lambda (B) - \frac{1}{n}}{\lambda (B)}.
\end{align*}
Hence we have obtained
\begin{equation} \label{eq:densityboundnu}
  5^{-d} \leq \liminf_{n \to \infty} \frac{\nu_n (B)}{\lambda
    (B)} \leq \limsup_{n \to \infty} \frac{\nu_n (B)}{\lambda
    (B)} \leq 5^d,
\end{equation}
for any ball $B$.

Since $\frac{\mathrm{d} \nu_n}{\mathrm{d} \lambda} \leq 2 \cdot 5
^d$ if $n \geq 2$, we have
\[
I_t (\nu_n) \leq 4 \cdot 5^{2d} \iint |x-y|^{-t} \, \mathrm{d} x
\mathrm{d} y = 4 \cdot 5^{2d} I_t (\lambda).
\]

We define the set $V_{n}$ by
\[
V_{n} = \bigcup_{j \in \mathscr{K}_n} U_j.
\]
The set $V_{n}$ is a subset of $\tilde{E}_n$, since $U_j \subset
B(x_j, r_j)$. We define new measures $\mu_n$ in the following
way. For each $j \in \mathscr{K}_n$, the mass of $\nu_n$ in $B
(x_j, r_j)$ is moved into $U_j$ and distributed uniformly. More
precisely, $\mu_n$ is defined by
\[
\mu_n = \sum_{j \in \mathscr{K}_n} \nu_n (B (x_j, r_j))
\frac{\lambda|_{U_j}}{\lambda (U_j)} =
\frac{1}{\lambda(\tilde{E}_n)} \sum_{j \in \mathscr{K}_n}
\frac{\lambda (B(x_j, r_j))}{\lambda (U_j)} \lambda|_{U_j}.
\]

By \eqref{eq:densityboundnu}, and the fact that $r_j \to 0$, we
immediately obtain that
\begin{equation} \label{eq:densityboundmu}
  5^{-d} \leq \liminf_{n \to \infty} \frac{\mu_n (B)}{\lambda
    (B)} \leq \limsup_{n \to \infty} \frac{\mu_n (B)}{\lambda
    (B)} \leq 5^d,
\end{equation}
holds for any ball $B$.  (Replacing $B(x_j, r_j)$ by $U_j$ moves
mass and could potentially move mass inside or outside of
$B$. Since $r_j \to 0$ this could only happen when $x_j$ is close
to the boundary, which can only happen for an increasingly small
proportion of $x_j$ as $n \to \infty$, since $r_j \to 0$.)

We shall now estimate $I_t (\mu_n)$.  From the definition of
$\mu_n$ follows immediately that
\[
I_t (\mu_n) = \sum_{j,k \in \mathscr{K}_n} \int_{U_k} \int_{U_j}
|x-y|^{-t} \, \mathrm{d} \mu_n (x) \mathrm{d} \mu_n (y).
\]
We consider two cases. Below, we write $B_j$ instead of $B(x_j,
r_j)$.

If $k = j$, then
\begin{multline*}
  \int_{U_k} \int_{U_j} |x-y|^{-t} \, \mathrm{d} \mu_n (x)
  \mathrm{d} \mu_n (y) = I_t (U_j) \biggl( \frac{\mu_n
    (U_j)}{\lambda (U_j)} \biggr)^2 \\ = I_t (U_j) \biggl(
  \frac{\nu_n (B_j)}{\lambda (U_j)} \biggr)^2 \leq 4 \cdot 5^{2d}
  I_t (U_j) \biggl( \frac{\lambda (B_j)}{\lambda (U_j)}
  \biggr)^2.
\end{multline*}

If $k \neq j$, then $B (x_j,r_j)$ and $B (x_k, r_k)$ are
disjoint. Hence $B (x_j, c r_j)$ and $B (x_k, c r_k)$ are
separated by at least a distance $(1-c) |x_j - x_k|$, and since
$U_k \subset B (x_k, c r_k)$ and $U_j \subset B (x_j, c r_j)$, we
have $d (U_k, U_j) \geq (1-c) |x_j - x_k|$.  We get the
inequalities
\begin{align*}
  \int_{U_k} \int_{U_j} |x-y|^{-t} \, \mathrm{d} \mu_n (x)
  \mathrm{d} \mu_n (y) &\leq d(U_k, U_j)^{-t} \mu_n (U_k) \mu_n
  (U_j)\\ &\leq (1-c)^{-t}|x_k - x_j|^{-t} \mu_n (U_k) \mu_n
  (U_j) \\ &= (1-c)^{-t} |x_k - x_j|^{-t} \nu_n (B_k) \nu_n (B_j)
  \\ &\leq 4 \cdot 5^{2d} (1-c)^{-t} |x_k - x_j|^{-t} \lambda
  (B_k) \lambda (B_j) \\ &\leq C_{t,d} J_t (B_k, B_j),
\end{align*}
where $C_{t,d}$ is a constant which only depends on $t$ and $d$.

We may now conclude that there is a constant $C$ such that $I_t
(\mu_n) < C$ for all $n$, provided
\begin{equation} \label{eq:condition}
  \frac{I_t (U_j) \lambda(B_j)}{\lambda (U_j)^2} < K
\end{equation}
holds for all $j$ and some constant $K$.
Indeed, it follows from the inequalities above that
\begin{align*}
  I_t (\mu_n) &= \sum_{j,k \in \mathscr{K}_n} \int_{U_k}
  \int_{U_j} |x-y|^{-t} \, \mathrm{d} \mu_n (x) \mathrm{d} \mu_n
  (y) \nonumber \\ & \leq 4 \cdot 5^{2d} \sum_{j \in
    \mathscr{K}_n} I_t (U_j) \biggl( \frac{\lambda (B_j)}{\lambda
    (U_j)} \biggr)^2 + \sum_{j,k \in \mathscr{K}_n} C_{t,d} J_t
  (B_k, B_j) \nonumber \\ & \leq 4 \cdot 5^{2d} \sum_{j \in
    \mathscr{K}_n} I_t (U_j) \biggl( \frac{\lambda (B_j)}{\lambda
    (U_j)} \biggr)^2 + C_{t,d} I_t (\nu_n) \nonumber \\ & \leq 4
  \cdot 5^{2d} \sum_{j \in \mathscr{K}_n} I_t (U_j) \frac{\lambda
    (B_j)}{\lambda (U_j)^2} \cdot \lambda (B_j) + 4 \cdot 5^{2d}
  C_{t,d} I_t (\lambda).
\end{align*}
Hence, if \eqref{eq:condition} holds, we may use
\eqref{eq:sumofmeasure} to conclude that
\begin{equation} \label{eq:energybound2}
  T_t (\mu_n) \leq 4 \cdot 5^{2d} C_d K + 4 \cdot 5^{2d} C_{t,d}
  I_t (\lambda),
\end{equation}
where the constants $C_d$, $K$ and $C_{t,d}$ do not depend on
$n$.

Finally, by \eqref{eq:densityboundmu} and
\eqref{eq:energybound2}, the assumptions in
Lemma~\ref{lem:frostman} are satisfied for the measures $\mu_n$
provided \eqref{eq:condition} is satisfied.
Lemma~\ref{lem:frostman} then implies that $\limsup_{n \to
  \infty} V_{n} \in \mathscr{G}^t (\mathbbm{T}^d)$ holds for all
$t$ which satisfies \eqref{eq:condition}. Since $\limsup_{n \to
  \infty} V_{n} \subset \limsup_{j \to \infty} U_j$, this implies
that $\limsup_{j \to \infty} U_j \in \mathscr{G}^s
(\mathbbm{T}^d)$ when $s$ is the supremum of the set of $t$ such
that \eqref{eq:condition} holds for all $j$.

\section{Appendix: Proof of Lemma~\ref{lem:frostman}} \label{sec:appendix}

Here we give a proof of Lemma~\ref{lem:frostman} following the
ideas in the papers \cite{perssonreeve} and \cite{persson2}.

We start with a lemma.

\begin{lemma} \label{lem:energylimit}
  Let $\mu_n$ and $C$ be as in Lemma~\ref{lem:frostman}, let $D$ be a
  dyadic cube and $t < s$. Then
  \[
  \limsup_{n \to \infty} \iint_{D \times D} |x - y|^{-t} \,
  \mathrm{d} \mu_n (x) \mathrm{d} \mu_n (y) \leq C^2 \iint_{D
    \times D} |x - y|^{-t} \, \mathrm{d}x \mathrm{dy}.
  \]
\end{lemma}

\begin{proof}
  Let $M_m = \{\, (x, y) \in D \times D : |x-y|^{-s} > m
  \,\}$. Then
  \[
  \iint_{M_m} |x-y|^{-t} \, \mathrm{d} \mu_n (x) \mathrm{d} \mu_n
  (y) \leq C \frac{s}{s - t} m^{t/s - 1}
  \]
  holds for all $n$ with $C$ as in Lemma~\ref{lem:frostman}
  \cite[Lemma~2.2]{perssonreeve}. Let $\varepsilon > 0$ and take
  $m$ so large that
  \[
  \iint_{M_m} |x-y|^{-t} \, \mathrm{d} \mu_n (x) \mathrm{d} \mu_n
  (y) < \varepsilon,
  \]
  holds for all $n$. Then
  \begin{multline*}
    \iint_{D \times D} |x - y|^{-t} \, \mathrm{d} \mu_n (x)
    \mathrm{d} \mu_n (y) \\ \leq \varepsilon + \iint_{D \times D}
    \min \{ |x - y|^{-t}, m^{t/s} \} \, \mathrm{d} \mu_n (x)
    \mathrm{d} \mu_n (y).
  \end{multline*}
  Since $(x,y) \mapsto \min \{ |x - y|^{-t}, m^{t/s} \}$ is
  continuous, we have by \eqref{eq:densitybounds}
  \begin{multline*}
    \limsup_{n\to \infty} \iint_{D \times D} \min \{ |x -
    y|^{-t}, m^{t/s} \} \, \mathrm{d} \mu_n (x) \mathrm{d} \mu_n
    (y) \\ \leq C^2 \iint_{D \times D} \min \{ |x - y|^{-t},
    m^{t/s} \} \, \mathrm{d}x \mathrm{d} y.
  \end{multline*}
  As $\varepsilon > 0$ is arbitrary, this proves the lemma.
\end{proof}

\begin{proof}[Proof of Lemma~\ref{lem:frostman}]
  Let $D$ be a fixed dyadic cube and take $t < s$. Then, if $n$
  is sufficiently large $\mu_n (D) > 0$, and the assumptions
  imply that
  \[
  \nu_n (A) = \int_{A \cap D} \biggl( \int_D |x-y|^{-t} \,
  \mathrm{d} \mu_n (x) \biggr)^{-1} \, \mathrm{d} \mu_n (y)
  \]
  defines a non-zero measure which is absolutely continuous with
  respect to $\mu_n$. The measure $\nu_n$ satisfies $\nu_n (D) =
  \nu_n (D \cap E_n)$.

  By Jensen's inequality, the measure $\nu_n$ satisfies
  \begin{align*}
    \nu_n (D) &= \int_D \biggl( \int_D |x-y|^{-t} \, \mathrm{d}
    \mu_n (x) \biggr)^{-1} \, \mathrm{d} \mu_n (y) \\ &= \int_D
    \biggl( \int_D |x-y|^{-t} \, \frac{\mathrm{d} \mu_n
      (x)}{\mu_n (D)} \biggr)^{-1} \, \frac{\mathrm{d} \mu_n
      (y)}{\mu_n (D)} \\ &\geq \biggl( \iint_{D \times D}
    |x-y|^{-t} \, \frac{\mathrm{d} \mu_n (x)}{\mu_n (D)} \,
    \frac{\mathrm{d} \mu_n (y)}{\mu_n (D)} \biggr)^{-1}.
  \end{align*}
  By Lemma~\ref{lem:energylimit}, we get that
  \begin{align*}
    \nu_n (D) &\geq \frac{C^{-2} \mu_n (D)^2}{2 \lambda (D)^2}
    \biggl( \iint_{D \times D} |x-y|^{-t} \, \frac{\mathrm{d}
      \lambda (x)}{\lambda (D)} \, \frac{\mathrm{d} \lambda
      (y)}{\lambda (D)} \biggr)^{-1} \\ &= \frac{C^{-2} \mu_n
      (D)^2}{\lambda (D)^2} \frac{\lambda (D)^2}{I_t (D)},
  \end{align*}
  if $n$ is sufficiently large. Hence, for large enough $n$, we
  have
  \[
  \nu_n (D) \geq \frac{C^{-2} \mu_n (D)^2}{2 \lambda (D)^2}
  \frac{\lambda (D)^2}{I_t (D)} \geq \frac{C^{-4}}{4}
  \frac{\lambda (D)^2}{I_t (D)} \geq C_0 |D|^t,
  \]
  for some constant $C_0$ which depends only on $d$ and $t$.
  
  For $U \subset D$ we have again by Jensen's inequality that
  \begin{align*}
    \nu_n (U) &= \int_{U} \biggl( \int_D |x-y|^{-t} \, \mathrm{d}
    \mu_n (x) \biggr)^{-1} \, \mathrm{d} \mu_n (y) \\ & \leq
    \int_U \biggl( \int_U |x-y|^{-t} \, \frac{\mathrm{d} \mu_n
      (x)}{\mu_n (U)} \biggr)^{-1} \, \frac{\mathrm{d} \mu_n
      (y)}{\mu_n (U)} \\ & \leq \int_U \int_U |x-y|^t \,
    \frac{\mathrm{d} \mu_n (x)}{\mu_n (U)} \frac{\mathrm{d} \mu_n
      (y)}{\mu_n (U)} \leq |U|^t.
  \end{align*}

  Suppose now that $\{D_k\}$ is a disjoint cover of $E_n \cap D$
  by dyadic cubes. Then if $n$ is large, we have by the estimates
  above that
  \[
  \sum_{k} |D_k|^t \geq \sum_{k} \nu_n (D_k) = \nu_n (D \cap E_n)
  = \nu_n (D) \geq C_0 |D|^t.
  \]
  This proves that $\mathscr{M}_t^\infty (E_n \cap D) \geq C_0
  |D|^t$ when $n$ is large and hence
  \[
  \liminf_{n \to \infty} \mathscr{M}_t^\infty (E_n \cap D) \geq
  C_0 |D|^t.
  \]
  By \cite{falconer2}, this implies that $\limsup E_n \in
  \mathscr{G}^s (\mathbbm{T}^d)$.
\end{proof}

\end{document}